\def\mytitle{Approximately Optimal Controllers for Quantitative\texorpdfstring{\\}~Two-Phase Reach-Avoid Problems on Nonlinear Systems}
\def\myname{Alexander Weber and Alexander Knoll}
\def\mykeywords{Symbolic Optimal Control, Abstractions, Reach-Avoid Problem, Parcel delivery, aircraft routing mission}
\def\arxivVersion{}
\ifx\arxivVersion\undefined
\else
\documentclass[journal,a4paper,onecolumn,12pt]{ieeeconf}
\usepackage{geometry}
\fi
\IEEEoverridecommandlockouts
\usepackage{cite}
\usepackage[english]{babel}
\usepackage{amsmath}
\usepackage{amssymb}

\usepackage{amsthm}
\usepackage{stmaryrd}
\usepackage{tikz}
\usetikzlibrary{calc,shapes,arrows,automata}
\usepackage{algorithm}
\usepackage{algpseudocode}
\algnewcommand\algorithmicinput{\textbf{Input:}}
\algnewcommand\Input{\item[\algorithmicinput]}
\usepackage[]{caption}
\swapnumbers

\newtheoremstyle{rem}{\topsep}{\topsep}{\normalfont}{0pt}{\bfseries}{.}{ }{\thmname{#1 }\thmnumber{#2}\thmnote{ \textup{(#3)}}}
\newtheorem{definition}{Definition}[section]
\newtheorem{theorem}{Theorem}[section]
\newtheorem{proposition}{Proposition}[section]

\theoremstyle{rem}
\newtheorem{example}[definition]{\pushQED{\qed}Example}

\makeatletter
\def\endexample{\popQED\@endtheorem}
\def\begriff#1%
{\@nomath\begriff\ifdim\fontdimen\@ne\font>\z@%
\textbf{#1}\else\textit{#1}\fi}
\def\intcc#1{\ensuremath{\left[#1\right]}}
\def\intco#1{\ensuremath{\left[#1\right[}}

\makeatother
\newcommand{\R}{\mathbb{R}}
\newcommand{\Z}{\mathbb{Z}}
\def\implies{\relax\ifmmode\mathrel{\Rightarrow}\else$\implies$ \fi}
\usepackage{paralist}
\usepackage{hyperref}
\usepackage{url}
\hypersetup{
colorlinks=true,%
breaklinks=true,%
pdfdisplaydoctitle=true,%
linkcolor={black},%
citecolor={black},
urlcolor={blue},
pdfstartview={FitV},%
pdftitle={\mytitle},
pdfauthor={\myname},
pdfsubject={Submitted to arXiv on \today.},
pdfkeywords={\mykeywords}
}
\graphicspath{{figures/}}
\ifx\arxivVersion\undefined
\else
\title{\bf \Large \mytitle}
\author{
\myname
\thanks{
The authors are with the
Munich University of Applied Sciences,
Dept. of Mechanical, Automotive and Aeronautical Eng.,
80335 M\"unchen, Germany. Corresponding author: A. Weber, \texttt{weber13@hm.edu}
}%
\thanks{This work has been supported by the German Federal Ministry of Education and Research (Project ARCUS). %
 }%
}%
\fi
\begin{document}
\maketitle
\ifx\arxivVersion\undefined
\else
\thispagestyle{empty}
\fi
\begin{abstract}
The present work deals with 
quantitative two-phase reach-avoid problems 
on nonlinear control systems. 
This class of optimal control problem requires the plant's state 
to visit two (rather than one) target sets in succession 
while minimizing a prescribed cost functional. 
As we illustrate, the naive approach, which subdivides 
the problem into the two evident classical reach-avoid tasks, 
usually does not result in an optimal solution. 
In contrast,
we prove that an optimal controller 
is obtained by consecutively solving two special 
quantitative reach-avoid problems.
In addition, we present a fully-automated method 
based on Symbolic Optimal Control 
to practically synthesize for the considered problem class
approximately optimal controllers for sampled-data nonlinear plants. 
Experimental results on parcel delivery and on an 
aircraft routing mission confirm the practicality of our method.
\end{abstract}
\section{Introduction}
\label{s:introduction}
Cyber-physical systems have become 
a highly promising engineering technology 
which requires for its implementation 
expertise from many different 
disciplines like 
electrical and mechanical engineering, 
physics or computer sciences. 
This new concept to control a physical entity 
is based on feedback laws, which should be 
designed such that the physical process is controlled 
in a ``reliable" and ``efficient" way \cite{LiuPengWangYaoLiu17}. 
Reliability and efficiency for closed-loop systems 
are each a challenging requirement in itself. 

Abstraction-based controller synthesis 
\cite{Tabuada09,
Schmuck15,
ReissigWeberRungger17,
Weber18} has been developed to tackle
the \emph{reliability} requirement for 
feedback controllers on nonlinear systems.
The said approach is a fully-automated scheme 
for synthesizing controllers that come with 
formal guarantees to meet the specification.
Symbolic Optimal Control 
\cite{RooMazo13,
ReissigRungger13,
ReissigRungger18} 
is a recent extension which aims at 
adding \emph{efficiency}
to the evolution of the closed loop within 
the meaning of minimizing a given cost functional.

Aforementioned synthesis approach contains 
two demanding steps with the first being the 
computation of a so-called discrete abstraction and 
the second being the solution of 
an auxiliary discrete problem.
Though experts in this field proudly emphasize that 
\emph{complex} control problems can be solved 
\cite{Schmuck15,ReissigWeberRungger17} 
concrete algorithms (for performing the second step) 
have been presented so far only 
for a few classes of control problems. 
Most works investigate safety problems 
\cite{Girard10,
MeyerGirardWitrant18,
HsuMajumdarMallikSchmuck18},
classical reach-avoid problems
\cite{
MazoTabuada11,
KhaledKimArcakZamani19,
MacoveiciucReissig19,
WeberKreuzerKnoll20} 
or reach-and-stay problems \cite{LiLiu18}. 
Some articles like 
\cite{
GrueneJunge08,
RunggerStursberg12,
EqtamiGirard18}
put emphasis on 
quantitative variants
of previous problems, that is, 
additionally consider optimization of costs.
So-called \emph{sequencing} of targets, 
which requires
the plant's state to 
visit $N$ target sets in a given order, 
is considered in \cite{ReissigWeberRungger17}
for the case $N=2$ and in
\cite{FainekosKressGazitPappas05,
FainekosGirardKressGazitPappas09}
for the case of integrator dynamics.

In this work, 
we consider a quantitative variant of sequencing 
for the case $N=2$, 
which we name 
\begriff{quantitative two-phase reach-avoid problem}. 
This specification requires the plant's state to 
visit two target sets in succession 
while minimizing a prescribed cost functional. 
(Obstacles should also be avoided.) 
In contrast to 
\cite{FainekosKressGazitPappas05,
FainekosGirardKressGazitPappas09,
ReissigWeberRungger17} 
we aim at solving the problem \emph{optimally} 
with respect to a rather arbitrary cost functional.
An example of such a problem 
already appeared 
in the preceding work \cite{WeberKreuzerKnoll20}.

To give a concrete example at this point, 
consider Fig.~\ref{fig:intro}: 
A delivery vehicle 
located in a Manhattan-like city 
has to make a delivery firstly in 
the green area (``Area 1") and 
then in the red area (``Area 2"). 
The question we address in this work 
is how to synthesize a controller 
fulfilling this specification and 
additionally operating for minimum cost 
in the sense of worst-case (or min-max) cost. 
The said problem could be naively solved 
by solving two ordinary (quantitative) 
reach-avoid problems as it was done in 
\cite{ReissigWeberRungger17} 
for the special case of time-optimality. 
More concretely, the first problem is 
to steer the vehicle from Area 1 to Area 2.
The second problem is to steer the vehicle from 
its initial state to the subset inside Area 1 on 
which the first controller is successful. 
Doing so, two controllers, each (approximately) optimal 
for the respective subproblem, 
have been synthesized,
which steer the vehicle along 
the highly suboptimal path 
indicated in Fig.~\ref{fig:intro}. 
The reason for the poor performance is that 
the manoeuvre to reach \mbox{Area 1} optimally
results in an orientation of the vehicle that
is disadvantageous for reaching \mbox{Area 2} next. 
\ifx\arxivVersion\undefined
\else
\begin{figure}
\begin{center}
\hspace*{1.3cm}
\input{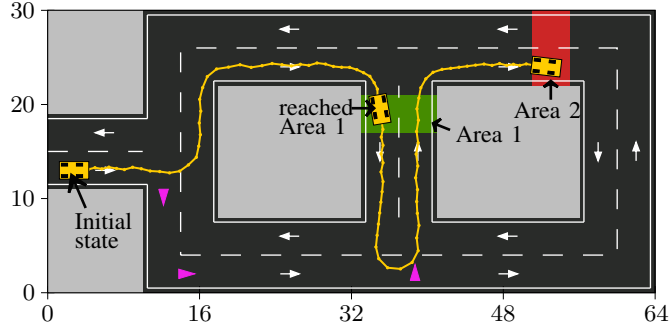}
\end{center}
\caption{\label{fig:intro}Motivating example: A delivery vehicle has to make a delivery firstly in Area 1 and then in Area 2. A naive solution of this two-phase reach-avoid problem results in the highly inefficient trajectory that is depicted. Directions of the optimal trajectory are indicated by the purple arrows. All details to this example are given in Section \ref{ss:examples:vehicle}.}
\end{figure}
\fi

In this paper, we prove that an optimal controller can be 
synthesized through solving two special quantitative
reach-avoid problems.
Moreover, we present a fully-automated method
based on aforementioned Symbolic Optimal Control to practically
synthesize approximately optimal controllers on 
sampled versions of plants governed by 
the $n$-dimensional differential inclusion
\begin{equation}
\label{e:system:cont}
\dot \xi(t) \in f(\xi(t),u(t)) + W.
\end{equation}
Here, $\xi$ is the $n$-dimensional state signal, 
$u$ is the input signal, 
$f$ is an ordinary function and 
$W$ is an $n$-dimensional set representing disturbances 
acting on the plant \cite{ReissigWeberRungger17}.

The paper is subsequently structured as follows.
Basic concepts and definitions used throughout the paper 
are introduced in Section \ref{s:prelim}. 
The main results on the synthesis 
of optimal controllers are 
given in Section \ref{s:main}. 
In Section \ref{s:examples} 
experimental results are
presented and in particular 
the introductory example 
is investigated in all detail. 
Conclusions and open problems 
are discussed in Section \ref{s:conclusions}.
\section{Preliminaries}
\label{s:prelim}
This work follows to a large extend the terminology of \cite{ReissigRungger13,ReissigRungger18}. 
For self-consistency, definitions used in this work are reviewed 
in this section, sometimes in an appropriately modified variant.
More concretely, 
the following five subsections include basic notation, 
the concept of transition system and closed loop,
the considered class of optimal control problems, 
the used notion of optimality and lastly a review
of the classical reach-avoid problem formulation.
\subsection{Basic notation}
\label{ss:notation}
The empty set is denoted by $\emptyset$. 
$A \setminus B$ denotes the difference of two sets $A$ and $B$. 
The cardinality
of a set $A$ will be denoted by $|A|$.
The symbols $\mathbb{Z}$, $\mathbb{R}$ and $\mathbb{Z}_+$, $\mathbb{R}_+$ 
stand for the set of integers and real numbers, respectively, 
and the respective subsets of non-negative elements.
For $a,b \in \R$ the closed and half-open interval 
with endpoints $a$ and $b$ is denoted by $\intcc{a,b}$ and $\intco{a,b}$, respectively. 
The symbols $\intcc{a;b}$ 
and $\intco{a;b}$ denote the discrete intervals, e.g. $\intcc{a;b} = \intcc{a,b} \cap \Z$.
For a function $f \colon A \to B$ the restriction of $f$ to a subset $C$ of $A$ 
is denoted by $f|_C$ \cite[p.4]{Hungerford74}. 
For $f\colon A \times B \to C$ and $b \in B$ the function $f(\cdot,b) \colon A \to C$ maps $a \in A$ to $f(a,b)$.
A set-valued map with domain $A$ and image the subsets of $B$ is denoted by $A \rightrightarrows B$ \cite{RockafellarWets09}. 
A set-valued map $f \colon A \rightrightarrows B$ is \begriff{strict} if $f(a) \neq \emptyset$ for all $a \in A$.
The symbol $f\circ g$ stands for the composition of the maps $f$ and $g$.
For sets $A$, $B$ the symbol $A^{B}$ stands for the set of all maps $B \to A$. 
An element of $A^{\Z_+}$ is called \begriff{signal}. 
For $f\in A^{\Z_+}$ and $\tau \in \Z_+$ the symbol $\sigma^{\tau}f$ stands 
for the backwards shifted signal $g\in A^{\Z_+}$ defined by $g(t) = f(t + \tau)$.
\subsection{System, controller and closed loop}
\label{ss:definitions}
Let $X$ and $U$ be non-empty sets. 
A \begriff{(transition) system} with \begriff{state space} $X$ and \begriff{input space} $U$ is a triple
\begin{equation}
\label{e:system}
(X,U,F),
\end{equation}
where $F\colon X \times U \rightrightarrows X$ is 
a strict set-valued map (called \begriff{transition function}).
System \eqref{e:system} is implicitly endowed with time-discrete dynamics 
implied by the difference inclusion 
\begin{equation}
\label{e:discrete:dynamics}
x(t+1) \in F(x(t),u(t)).
\end{equation}
The \begriff{behaviour} of 
a system $S$ of the form \eqref{e:system} 
\begriff{initialized at $p \in X$} 
is the set of all signal pairs 
$(u,x) \in (U \times X)^{\Z_+}$ such that 
$x(0) = p$ and
\eqref{e:discrete:dynamics} holds for all $t \in \Z_+$. 
This set is denoted by $\mathcal{B}_p(S)$.

The concept of a closed loop is formalized 
by means of a controller and 
the respective closed-loop behaviour:
A \begriff{controller} 
is a strict set-valued map \cite[Eq.~4]{ReissigRungger13}
\begin{equation*}
\label{e:controller}
\mu \colon \bigcup_{T\in \Z_+} (X^{\intcc{0;T}} \times U^{\intco{0;T}}) \rightrightarrows U \times \{0,1\},
\end{equation*}
where the second component of the image reports 
whether the controller is in operation ('$0$') or 
has stopped operating ('$1$') \cite[Sec.~III.A]{ReissigRungger13}.
The set of all such maps is denoted by $\mathcal{F}(X,U)$. 
A controller $\mu \in \mathcal{F}(X,U)$ is called \begriff{memoryless} 
if a representation as a map 
$\mu \colon X \rightrightarrows U \times \{0,1\}$ is possible. 
The subset of memoryless controllers is denoted by $\mathcal{F}_0(X,U)$. 

The \begriff{closed-loop behaviour} 
of a system $S$ of the form \eqref{e:system} 
interconnected with $\mu \in \mathcal{F}(X,U)$ 
and initialized at $p \in X$ 
is the set of all signals $(u,v,x)$ 
that satisfy $(u,x) \in \mathcal{B}_p(S)$ and\looseness=-1 
\[
(u(t),v(t)) \in \mu(x|_{\intcc{0;t}},u|_{\intco{0;t}})
\] 
for all $t \in \Z_+$. 
The symbol $\mathcal{B}_p(\mu \times S)$ stands for this set.
\subsection{Optimal control problem}
\label{ss:ocp}
As explained later, this work considers slightly 
more general optimal control problems than formulated in \cite{ReissigRungger18}. 
Common with \cite{ReissigRungger18} 
is that the closed loop is \begriff{leavable}, 
i.e. the controller may stop 
its operation at any time. Moreover,
stopping will be mandatory.
Also, optimality is meant 
in terms of an accumulation of two kinds of costs that accrue
for the operation of the closed loop. Costs may be infinite, 
which allows to formulate hard state and input constraints. 
Firstly, as usually in optimal control,
there is a \begriff{running cost}
\begin{equation}
\label{e:runningcost}
g \colon X \times X \times U \to \mathbb{R}_+ \cup \{ \infty \}
\end{equation}
accruing between two consecutive points in time. 
Secondly, the closed-loop trajectory 
is rated once at stopping time by a \begriff{trajectory cost} 
\begin{equation}
\label{e:trajectorycost}
G \colon \bigcup_{T \in \Z_+} X^{\intcc{0;T}} \to \mathbb{R}_+ \cup \{ \infty \}.
\end{equation}
Altogether, for a signal pair $(u,x) \in \mathcal{B}_p(S)$ and a 
stopping signal $v \in \{0,1\}^{\Z_+}$ 
we assign a \begriff{total cost} 
\begin{subequations}
\label{e:costfunctional}
\begin{equation}
\label{e:costfunctional:declaration}
J \colon (U \times \{0,1\} \times X)^{\Z_+} \to \R_+ \cup \{ \infty \}
\end{equation}
defined by %
\begin{equation}
\label{e:costfunctional:definition}
J(u,v,x) = \sum_{t = 0}^{T-1} g(x(t),x(t+1),u(t)) + G(x|_{\intcc{0;T}})
\end{equation}
\end{subequations}
for $T = \min v^{-1}(1)$ if $v \neq 0$, and $J(u,v,x) = \infty$ if \mbox{$v = 0$}. 
(Note that the last case formalizes the requirement for stopping, 
which occurs at the first 0-1 edge in $v$.)

By putting together the previous objects 
we arrive at the subsequent definition.
\begin{definition}
\label{def:ocp}
Let $(X,U,F)$ be a system, 
and $G$ and $g$ be as in \eqref{e:trajectorycost} and \eqref{e:runningcost}, 
respectively. The tuple
\begin{equation}
\label{e:ocp}
(X, U, F, G, g)
\end{equation}
is called \begriff{optimal control problem}.
\end{definition}
Definition \ref{def:ocp} agrees with \cite[Def.~III.3]{ReissigRungger18}
if the trajectory cost in \eqref{e:trajectorycost} possesses the special form 
\begin{equation}
\label{e:terminalcost}
G \colon X \to \mathbb{R}_+ \cup \{ \infty \}.
\end{equation} 
In this case, the map in \eqref{e:terminalcost} is called \begriff{terminal cost} \cite{ReissigRungger18}.
\subsection{Solution of an optimal control problem}
In order to introduce the notion of solution 
of an optimal control problem
we introduce the following quantity \cite{ReissigRungger18}.
\begin{definition}
Let $\Pi$ be an optimal control problem 
of the form \eqref{e:ocp} and 
let $J$ be as in \eqref{e:costfunctional}. 
The \begriff{closed-loop performance} of $\Pi$
is the map assigning $(p,\mu) \in X \times \mathcal{F}(X,U)$
to 
\begin{equation}
\label{e:closedloopperformance}
\sup_{(u,v,x) \in \mathcal{B}_p(\mu \times S)} J(u,v,x),
\end{equation}
where $S= (X,U,F)$.
\end{definition}
Having introduced the quantity \eqref{e:closedloopperformance} 
the concept of \begriff{value function} can be formulated 
as follows \cite[Sec.~III.A,VI.C]{ReissigRungger18}.
\begin{definition}
Let $\Pi$ be an optimal control problem of the form \eqref{e:ocp}. 
Denote by $L(p,\mu)$ the value \eqref{e:closedloopperformance}.
The \begriff{value function} of $\Pi$ is the map 
$V \colon X \to \mathbb{R}_+ \cup \{\infty\}$ defined by 
\begin{equation}
V(p) = \inf_{\mu \in \mathcal{F}(X,U)} L(p,\mu).
\end{equation}
A controller $\mu \in \mathcal{F}(X,U)$ is \begriff{optimal} for $\Pi$ 
if $V = L(\cdot,\mu)$.
Such a pair $(V,\mu)$ is called \begriff{solution} of $\Pi$. 
\end{definition}
So, an optimal controller minimizes the worst-case costs.
\subsection{The classical reach-avoid problem}
The classical (qualitative) reach-avoid problem is 
the problem of steering the state signal of a system $S$ of the form 
\eqref{e:system} into a \begriff{target set} $A \subseteq X$ 
while avoiding an \begriff{obstacle set} $O \subseteq X$. 
This specification can be easily formulated as 
an optimal control problem \eqref{e:ocp}, 
cf. \cite[Ex.~III.5]{ReissigRungger18}: 
\\
Define the trajectory cost $G$ as 
a terminal cost \eqref{e:terminalcost} by 
\begin{equation}
\label{e:reachavoid:terminalcost}
G(p) = \begin{cases} 
0, & \text{if } p \in A \\
\infty, & \text{otherwise}
\end{cases}
\end{equation}
and define the running cost $g$ in \eqref{e:runningcost} by 
\begin{equation*}
g(x,y,u) = \begin{cases} 
0, & \text{if } x \in X \setminus O \\
\infty, & \text{otherwise}
\end{cases}.
\end{equation*}
An optimal controller is ``successful" for every state signal 
that is initialized at a point in $X \setminus V^{-1}(\infty)$, 
where $V$ stands for the value function of the 
defined optimal control problem.

Clearly, a quantitative version of a reach-avoid problem is obtained 
when allowing $G$ and $g$ to be non-constant instead of $0$ 
on the respective subsets. 
(Examples can be found in 
\cite{RunggerStursberg12,WeberKreuzerKnoll20}.) 
We formalize this version 
in the next definition, 
which already indicates 
that for the upcoming contributions 
the key object will be the trajectory cost 
(rather than the running cost)
of an optimal control problem. 
\begin{definition}
\label{def:reachavoid}
Let $\Pi$ be of the form \eqref{e:ocp} such that 
$G$ satisfies \eqref{e:reachavoid:terminalcost} with $G_0(p)$ in place of $0$, 
where $G_0 \colon X \to \mathbb{R}_+ \cup \{ \infty\}$ and 
$A \subseteq X$ is non-empty.
Then $\Pi$ is called \begriff{(quantitative) reach-avoid problem} 
associated with $A$ and $G_0$.
\end{definition}

We extend in this work Definition \ref{def:reachavoid} 
to a class of problems
that contains an additional 
intermediate goal - a ``stopover set", as explained below.
\section{Quantitative 2-phase reach-avoid problems}
\label{s:main}
We divide the presentation of our main contributions into three parts. 
Two-phase reach-avoid problems are rigorously 
defined in Section \ref{ss:main:definition}. 
In Section \ref{ss:main:statements} we deduce the structure of an optimal 
controller for this class of problems. 
A constructive method to synthesize 
approximately optimal controllers on sampled-data systems 
is given in Section \ref{ss:main:approximatesolution}.
\subsection{Problem definition}
\label{ss:main:definition}
We are going to investigate 
the following class of 
optimal control problem.

\begin{definition}
\label{def:tworeachavoid}
Let $\Pi$ be an optimal control problem 
of the form \eqref{e:ocp} such that
$G$ is defined by
\begin{equation}
\label{e:tworeachavoid}
G(x|_{\intcc{0;t}}) = \begin{cases} 
G_0(x(t)), & \text{if } 
\exists_{s \leq t} : x(s) \in A_1 \wedge x(t) \in A_2 \\
\infty, &  \text{otherwise}
\end{cases},
\end{equation}
where $G_0 \colon X \to \mathbb{R}_+ \cup \{ \infty\}$ and 
$A_1, A_2 \subseteq X$ are non-empty sets.
Then $\Pi$ is called \begriff{(quantitative) two-phase reach-avoid problem} associated with $A_1$, $A_2$ and $G_0$.
\end{definition} 
Thus, for a two-phase reach-avoid problem stopping may occur only if firstly $A_1$ is visited and then $A_2$ is reached. 

\subsection{Structure of an optimal controller}
\label{ss:main:statements}

Our first result below defines a controller 
for a two-phase reach-avoid problem
whose closed-loop performance is equal to 
the one of a (memoryless) controller 
for a special 
classical reach-avoid problem.
The defined controller 
structure will turn out to be 
appropriate for the desired 
optimal controller.

\begin{proposition}
\label{p:closedloopperformance}
Let $\Pi$ be a 
two-phase reach-avoid problem 
of the form \eqref{e:ocp} 
associated with $A_1$, $A_2$ and $G_0$.
Let 
\mbox{$\Pi_2 = (X,U,F,G_2,g)$}
be the reach-avoid problem associated with $A_2$ and $G_0$.
Let $\mu_2 \in \mathcal{F}_0(X,U)$ and let
$
\Pi_1=(X,U,F,G_1,g)
$ 
be the reach-avoid problem associated with $A_1$ and $L_2(\cdot, \mu_2)$,
where $L_2$ denotes the closed-loop performance of $\Pi_2$. 
Let $\mu_1 \in \mathcal{F}_0(X,U)$ and consider
$\mu \in \mathcal{F}(X,U)$ defined by
\begin{equation}
\label{e:closedloopperformance:controller}
\mu(x|_{\intcc{0;t}},u|_{\intco{0;t}}) = \begin{cases} 
\mu_1(x(t)), & \text{if } \ w(\intcc{0;t}) = \{0\}   \\
\mu_2(x(t)), & \text{otherwise}
\end{cases},
\end{equation}
where $w := (\mu_1 \circ x)_2$ is the stopping signal of $\mu_1$.
Then \[L(p,\mu) = L_1(p,\mu_1)\] for all $p \in X$, 
where $L$ and $L_1$ denote the closed-loop performance of $\Pi$ and $\Pi_1$, respectively.
\end{proposition}
By virtue of previous proposition 
we will be able to conclude the following 
central theorem of this work:
\begin{theorem}
\label{th:main}
Let $\Pi$, $\Pi_1$, $\Pi_2$ and $\mu$, $\mu_1$, $\mu_2$ 
be as in Proposition \ref{p:closedloopperformance}. 
Additionally, 
assume that $\mu_i$
is optimal for $\Pi_i$ for $i \in \{1,2\}$.
Then $\mu$ is optimal for $\Pi$. 
\end{theorem}
\begin{proof}[Proof of Proposition \ref{p:closedloopperformance}]
Let $p \in X$. Denote by $J$, $J_1$, $J_2$ the total cost \eqref{e:costfunctional} 
as defined for $\Pi$, $\Pi_1$, $\Pi_2$, respectively.\\
Firstly, we prove that for all signals
$(u,v,x) \in \mathcal{B}_p(\mu \times S)$, where $S=(X,U,F)$, it holds
\begin{equation}
\label{e:closedloopperformance:step:1}
J(u,v,x) \leq L_1(p,\mu_1)
\end{equation}
if $L_1(p,\mu_1)$ is finite. 
Under the latter assumption, $\tau := \inf w^{-1}(1)$ and 
$T := \inf v^{-1}(1)$ are finite, where $w := (\mu_1 \circ x)_2$.
Then, using \eqref{e:costfunctional:definition}  
as defined for $\Pi$ and $\Pi_2$  
yields%
\ifx\arxivVersion\undefined
\else
\begin{align*}
J(u,v,x) &= {\sum_{t=0}^{\tau-1} g(x(t),x(t+1),u(t))} 
+ \sum_{t=\tau}^{T-1} g(\ldots)
+ G_0(x(T)) \\
&=
\sum_{t=0}^{\tau-1} g(x(t),x(t+1),u(t)) 
+ J_2(\sigma^{\tau}u,\sigma^{\tau}v,\sigma^{\tau}x).
\end{align*}
\fi
(Above, the ellipsis stands for the evident arguments to $g$.)
Using \eqref{e:closedloopperformance} as defined for $\Pi_2$ 
we see that the last term in the last sum is not greater than 
$L_2(x(\tau),\mu_2)$, so 
we conclude \eqref{e:closedloopperformance:step:1} 
using \eqref{e:costfunctional:definition} and \eqref{e:closedloopperformance} 
as defined for $\Pi_1$.\\
Next, assume $L(p,\mu) = \infty$. 
Then, for every $M \geq 0$
there exists a signal $(u,v,x) \in \mathcal{B}_p(\mu \times S)$
such that $J(u,v,x) \geq M$.
Assume that, however, $L_1(p,\mu_1)$ is finite. 
Then, by \eqref{e:closedloopperformance:step:1}, $M \leq L_1(p,\mu_1)$, a contradiction. 
So, $L(p,\mu) = L_1(p,\mu)$.\\
Now, assume 
$L(p,\mu) < \infty$. 
Obviously, $L_1(p,\mu_1) < \infty$. 
Then, for $\varepsilon > 0$ there exists $(u,v,x) \in \mathcal{B}_p(\mu \times S)$ such that 
$L(p,\mu) < J(u,v,x) + \varepsilon < L_1(p,\mu_1) + \varepsilon$ by \eqref{e:closedloopperformance:step:1}. 
The limit $\varepsilon \to 0$ implies $L(p,\mu) \leq L_1(p,\mu_1)$. \\
For the reverse inequality, let $\varepsilon > 0$ and 
note that
there exist $v \in \{0,1\}^{\Z_+}$ and $(u,x) \in \mathcal{B}_p(S)$ 
such that the inequalities
\begin{align*}
&L_1(p,\mu_1) - J_1(u,v,x) < \varepsilon, \\
&L_2(x(\tau),\mu_2) - J_2(\sigma^{\tau}u,\sigma^{\tau}v, \sigma^{\tau}x ) < \varepsilon 
\end{align*}
hold, where $\tau := \min v^{-1}(1)$.
It follows that 
\begin{align*}
& L_1(p,\mu_1) - J(u,v,x) < \varepsilon + J_1(u,v,x) - J(u,v,x) = \\
&=\varepsilon + L_2(x(\tau),\mu_2) - J_2(\sigma^{\tau}u,\sigma^{\tau}v,\sigma^{\tau}x ) < 2 \varepsilon.
\end{align*}
We deduce 
\begin{align*}
0 &\leq L(p,\mu) - J(u,v,x) \\
& = L_1(p,\mu_1) - J(u,v,x) + L(p,\mu) - L_1(p,\mu_1) \\
&\leq 2\varepsilon + L(p,\mu) - L_1(p,\mu_1). 
\end{align*}
Using the limit $\varepsilon \to 0$ we conclude $L(p,\mu) \geq L_1(p,\mu_1)$.
\end{proof}
\begin{proof}[Proof of Theorem \ref{th:main}]
Let $p \in X$ and $\mu^*$ be an optimal controller for $\Pi$. 
By Proposition \ref{p:closedloopperformance}
it is enough to verify 
\begin{equation}
\label{e:proof:main:1}
L(p,\mu^\ast) \geq L_1(p,\mu_1).
\end{equation}
To this end, 
let $\bar \mu^\ast \in \mathcal{F}(X,U)$ be such that 
its image coincides with $\mu^\ast$ in the first component 
but $\bar \mu^\ast( a|_{\intcc{0;t}}, b|_{\intco{0;t}} )_2 = \{1\}$ 
if $a(t) \in A_1$ for any $a$, $b$ and $t$. 
Let $\varepsilon > 0$ and $(u,\bar v, x) \in \mathcal{B}_p(\bar \mu^\ast \times S)$, where $S= (X,U,F)$, satisfy
\begin{equation}
\label{e:proof:main:2}
|L_1(p,\bar \mu^\ast) - J_1(u,\bar v, x) | < \varepsilon.
\end{equation}
Let $\tau = \min \bar v^{-1}(1)$ and $q = x(\tau)$.
Then $L_2(q, \mu_2) = L(q,\mu^\ast)$ 
by the optimality of $\mu_2$ and 
as $q \in A_1$. 
So, without loss of generality 
$|L_2(q,\mu_2) - J(\sigma^{\tau}u, \sigma^{\tau}v,\sigma^{\tau}x) | < \varepsilon$, where $v \in \{0,1\}^{\Z_+}$ satisfies
$v(\tau +t) \in \mu^\ast((\sigma^{\tau}x)|_{\intcc{0;t}},(\sigma^{\tau}u)|_{\intco{0;t}})_2$ for $t \in \Z_+$ and $v(\intco{0;\tau}) = \{0\}$.
It follows
\begin{equation}
\label{e:proof:main:3}
|J_1(u,\bar v, x) - J(u,v,x)| < \varepsilon.
\end{equation}
By the triangle inequality applied to
\eqref{e:proof:main:2} and \eqref{e:proof:main:3} 
it follows
\begin{equation}
\label{e:proof:main:4}
| L_1(p,\bar \mu^\ast) - J(u,v,x) | < 2 \varepsilon.
\end{equation}
Since $L_1(p, \mu_1) \leq L_1(p ,\bar \mu^\ast)$ by the optimality of $\mu_1$
the combination with \eqref{e:proof:main:4} and 
the definition of $v$ yields 
\begin{equation*}
L_1(p,\mu_1) \leq J(u,v,x) + 2 \varepsilon \leq L(p,\mu^\ast) + 2\varepsilon.
\end{equation*} 
Using the limit $\varepsilon \to 0$ we conclude \eqref{e:proof:main:1}.
\end{proof}
\subsection{Fully-automated approximate solution}
\label{ss:main:approximatesolution}
The findings presented previously do not directly
imply a constructive method for obtaining an optimal controller \eqref{e:closedloopperformance:controller}
since optimal controllers $\mu_1$ and $\mu_2$ need to be known.
Below, a constructive method to approximately solve
quantitative two-phase reach-avoid problems 
on sampled versions of plants governed 
by the differential inclusion \eqref{e:system:cont} 
is presented.

To begin with, from Proposition \ref{p:closedloopperformance} 
and Theorem \ref{th:main}
the synthesis procedure 
given in Fig.~\ref{alg:twophase} can be deduced, which on success
results in an optimal controller for the input problem. 
\begin{figure}[h]
\hrule{}
\vspace{1ex}
\begin{algorithmic}[1]
\Input{$\Pi = (X,U,F,G,g)$ in Prop.~\ref{p:closedloopperformance}.}
\State{\label{alg:twophase:firstproblem}$(V_2,\mu_2) \gets$ \text{Solution of} $\Pi_2$ in Prop.~\ref{p:closedloopperformance}}
\If{$V_2(A_1) = \{\infty\}$}
\State{\Return{$\emptyset$}\hfill{}/\!/~No non-trivial solution}
\EndIf{}
\State{\label{alg:twophase:secondproblem}$(V_1,\mu_1) \gets$ \text{Solution of} $\Pi_1$ in Prop.~\ref{p:closedloopperformance}}
\State{\Return{$(\mu_1,\mu_2)$}\hfill{}/\!/~Success}
\end{algorithmic}
\vspace{1ex}
\hrule{}
\vspace{1ex}
\caption{\label{alg:twophase}Procedure to synthesize $\mu$ in \eqref{e:closedloopperformance:controller} such that 
$\mu$ is optimal for the quantitative two-phase reach-avoid problem $\Pi$.}
\end{figure}%
\ifx\arxivVersion\undefined
\vspace{-2ex}%
\else
\fi
In the case that $X$ and $U$ of the input data 
in Fig.~\ref{alg:twophase} are finite sets, 
concrete algorithms 
to perform lines \ref{alg:twophase:firstproblem} and 
\ref{alg:twophase:secondproblem} have been formulated 
in \cite{ReissigRungger18,MacoveiciucReissig19} 
(based on Dijkstra's algorithm) 
and 
\cite{WeberKreuzerKnoll20} 
(based on the Bellman-Ford algorithm).
These concrete algorithms have been developed within 
the already mentioned framework of {Symbolic Optimal Control} (SOC) \cite{ReissigRungger18} and whose motivations were 
to efficiently solve optimal control problems with terminal cost (rather than trajectory cost) on sampled-data plants.

In what follows, we show how to utilize SOC to deduce
the announced constructive method. We begin with a brief review
on the synthesis method of SOC. 
The key idea is to lift the original 
optimal control problem $\Pi = (X,U,F,G,g)$ to 
an \emph{abstract} optimal control
problem $\Pi'$ with 
\emph{finite} state space $X'$ and 
\emph{finite} input set $U'$. 
In fact, $X'$ is chosen as a finite cover of $X$ and $U'$ as a subset of $U$. 
The fundamental results of 
\begriff{symbolic} (or \begriff{abstraction-based}) controller synthesis 
then allow to conclude an approximate
solution for the original problem $\Pi$ (in case of a terminal cost \eqref{e:trajectorycost}) by \begriff{controller refinement}.
Specifically, the refined controller 
is the interconnection of the ``abstract" controller $\mu'$ synthesized for $\Pi'$ and a quantizer \mbox{(``A/D converter")}
\begin{equation}
\label{e:quantizer}
Q \colon X \rightrightarrows X', \quad \Omega \in Q(p) :\Leftrightarrow p \in \Omega.
\end{equation}
See Fig.~\ref{fig:symboliccontrol}. The precise statement is 
as follows \cite[Th.~V.6, Prop.~VI.2]{ReissigRungger18}. 
\ifx\arxivVersion\undefined
\else
\begin{figure}
\centering
\usetikzlibrary{arrows}
\begin{tikzpicture}[scale=.75, every node/.style={scale=.8}]
\draw[thick]  (-3.75,3.25) rectangle (-1.5,2.5) node[pos=.5] {$(X,U,F)$};
\draw[thick]   (0.375,3.25) node (v2) {} rectangle (2.625,2.5) node[pos=.5] {$(X',U',F')$};
\draw[-latex] (-1.5,2.875) -- (-1.25,2.875) -- (-1.25,1.125) -- (-1.5,1.125);
\draw  (-2.25,1.375) rectangle (-1.625,0.875) node [pos=.5] (v4) {$Q$};
\draw  (-3.625,1.375) rectangle (-2.75,0.875) node[pos=.5] {$\mu'$};
\draw[thick,densely dotted] (-2.625,2) circle (0.3) node {$\Pi$};
\draw[thick,densely dotted] (1.5,2) circle (0.3) node {$\Pi'$};
\draw[-latex] (-2.25,1.125) -- (-2.75,1.125);
\draw[-latex] (-3.75,1.125) -- (-4,1.125) -- (-4,2.875) -- (-3.75,2.875);
\node (v1) at (-2.625,3.25) {};
\node (v4) at (1.5,3.25) {};
\draw[gray!60!white,line width=2][->]  (v1) edge [ bend angle=12,bend left] (v4);
\node at (-0.625,3.75) {Abstraction};
\draw  (0.375,1.5) rectangle (2.625,0.75) node[pos=.5] {$\mu'$};
\draw[-latex] (2.625,2.875) -- (2.875,2.875) -- (2.875,1.125) -- (2.625,1.125);
\draw[-latex] (0.375,1.125) -- (0.125,1.125) -- (0.125,2.875) -- (0.375,2.875);
\node (v3) at (0.375,0.75) {};
\node (v5) at (1.5,0.75) {};
\node (v6) at (-2.625,0.75) {};
\draw[gray!60!white,line width=2,densely dashed][->]  (v5) edge [ bend angle=12,bend left] (v6);
\node at (-0.625,0.25) {Refinement};
\node (v7) at (-1.5,0.75) {};
\draw[densely dashed]  (-3.75,1.5) rectangle (v7);
\end{tikzpicture}
\caption{\label{fig:symboliccontrol}Principle of symbolic (or abstraction-based) controller synthesis.}
\end{figure}
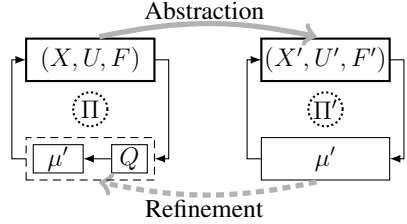 
\fi
\begin{theorem}[\!\!\cite{ReissigRungger18}]
\label{th:soc}
Let 
\begin{equation*}
\Pi = (X, U, F, G, g) \ \text{ and } \ \Pi' = (X', U', F', G', g')
\end{equation*}
be optimal control problems 
such that $G \colon X \to \mathbb{R}_+ \cup \{ \infty \}$, 
$G' \colon X' \to \mathbb{R}_+ \cup \{ \infty \}$
and the following holds:
\begin{enumerate}[(i)]
\item\label{th:soc:i} $X'$ is a cover of $X$ by non-empty sets;
\item\label{th:soc:ii} $U' \subseteq U$,
\end{enumerate}
and whenever $x \in \Omega \in X'$, $x' \in \Omega' \in X'$, $u \in U'$:
\begin{enumerate}[(i)]
\setcounter{enumi}{2}
\item\label{th:soc:iii} $\Omega ' \cap F(\Omega, u) \neq \emptyset \Rightarrow \Omega' \in F'(\Omega,u)$;
\item\label{th:soc:iv} $G(x) \leq G'(\Omega)$;
\item\label{th:soc:v} $g(x,x',u) \leq g'(\Omega,\Omega',u)$.
\end{enumerate}
Let $\mu ' \in \mathcal{F}_0(X',U')$ and let $Q$ be as in \eqref{e:quantizer}.
Then 
\[L(p,\mu' \circ Q) \leq \sup\{L'(\Omega,\mu' ) \mid p \in \Omega\}\] 
for all $p \in X$, 
where $L$ and $L'$ denote 
the closed-loop performance of $\Pi$ and $\Pi'$ respectively.
\end{theorem}
Furthermore, the results of \cite{ReissigRungger18} 
provide a statement on the distance of $L$ to $L'$ 
as well as a convergence statement. 
A computational technique to efficiently 
implement \eqref{th:soc:iii} 
under certain assumptions on 
the involved transition system $(X,U,F)$ is given in \cite{ReissigWeberRungger17,Weber18}.

Having previous results at hand, 
the synthesis technique to solve 
two-phase reach-avoid problems 
on sampled systems is 
as given in Fig.~\ref{alg:approximate}. 
Loosely speaking, the idea is 
to apply Fig.~\ref{alg:twophase} to the abstract problem $\Pi'$ and 
to follow the method depicted in Fig.~\ref{fig:symboliccontrol}.
This results in an approximately optimal controller, whose structure is depicted in
Fig.~\ref{fig:controller} and whose formal properties are provided by next theorem.
\begin{figure}[h]
\hrule{}
\vspace{1ex}
\begin{algorithmic}[1]
\Input{Two-phase reach-avoid problem $(X,U,F,G,g)$ associated with $A_1$, $A_2$ and $G_0$}
\State{\label{alg:approximate:abstraction}Define $X'$, $U'$, $F'$ satisfying \eqref{th:soc:i}--\eqref{th:soc:iii} of Theorem \ref{th:soc}.}
\State{Define $G'$ by \eqref{e:tworeachavoid} with 
\begin{align*}
&A'_i = \{\Omega \in X' \mid \Omega \subseteq A_i \}, \ i \in \{1,2\}, \\
&G_0'(\Omega) \geq \sup_{x \in \Omega}\nolimits G_0(x).
\end{align*}
in place of $A_1$, $A_2$ and $G_0$, respectively.}
\State{Define $g'$ satisfying \eqref{th:soc:v} in Theorem \ref{th:soc}.}
\State{\label{alg:approximate:apply}Apply the procedure in Fig.~\ref{alg:twophase} to $(X',U',F',G',g')$}
\If{line \ref{alg:approximate:apply} returns a non-trivial solution}
\State{Define $Q$ by \eqref{e:quantizer}}
\State{Define $\mu_Q \in \mathcal{F}(X,U)$ by \eqref{e:closedloopperformance:controller} with 
$\mu_1 ' \circ Q$ and $\mu_2' \circ Q$ 

in place of $\mu_1$ and $\mu_2$, respectively,
where $(\mu_1',\mu_2')$ 

is the return value of line \ref{alg:approximate:apply}.}
\State{\Return{$\mu_Q$}\hfill{}/\!/~Success}
\Else{}
\State{\Return{$\emptyset$}\hfill{}/\!/~Synthesis failed}
\EndIf{}
\end{algorithmic}
\vspace{1ex}
\hrule{}
\vspace{1ex}
\caption{\label{alg:approximate}Procedure to synthesize an approximately optimal controller for the quantitative two-phase reach-avoid problem $(X,U,F,G,g)$. 
}
\end{figure}
\ifx\arxivVersion\undefined
\else
\begin{figure}
\centering
\usetikzlibrary{arrows}
\ifx\arxivVersion\undefined
\begin{tikzpicture}[scale=.75, every node/.style={scale=.8}]
\else 
\begin{tikzpicture}[scale=1., every node/.style={scale=.9}]
\fi
\draw[thick]  (-5,2.5) rectangle (-0.5,1.5);
\draw[thick]  (-5,1) rectangle (-0.5,0);
\draw[thick]  (0.625,1.75) rectangle (2.75,0.75);
\coordinate (v1) at (-0.5,2) {};
\coordinate (v2) at (0,2) {};
\coordinate (v4) at (0,0.5) {};
\coordinate (v5) at (-0.5,0.5) {};
\coordinate (v3) at (0,1.25) {} ;
\coordinate (v6) at (0.625,1.25) {};
\draw[-latex]  (v2) -- (v1);
\draw[-latex] (v2) -- (v3) -- (v4) -- (v5);
\draw[]  (v3) -- (v6);
\coordinate (v11) at (-5,1.75) {} {} {};
\coordinate (v12) at (-5.25,1.75) {} {};
\coordinate (v13) at (-5,0.75) {} {} {} {} {};
\coordinate (v14) at (-5.25,0.75) {} {};
\draw (v11) -- (v12);
\draw[] (v13) -- (v14);
\coordinate (v19) at (-6.5,1.25) {} {};
\coordinate (v18) at (-7,2.75) {} {} {} {};
\coordinate (v15) at (2.75,1.25) {};
\coordinate (v16) at (3.5,1.25) {} ;
\coordinate (v17) at (3.5,2.75) {};
\draw[-latex,thick] (v17) -- (v16) -- (v15);
\node at (-2.75,2) {\large{}$\mu_1' \colon X' \rightrightarrows U' \times \{0,1\}$};
\node at (-2.75,0.5) {\large{}$\mu_2' \colon X' \rightrightarrows U'\times \{0,1\}$};
\node at (1.75,1.25) {$Q\colon X \rightrightarrows X'$};
\node[align=center] at (1.75,0.25) {\small{}Quantizer\\\small{}(A/D converter)};
\node at (0.25,1.5) {$\Omega$};
\node at (3.75,2) {$x$};
\node at (-7.25,2) {$u$};
\draw [dashed,thick](v19) -- (v12);
\draw [dashed,thick](v19) -- (v14);
\draw [black,fill=black] (v14) ellipse (0.05 and 0.05);
\draw [black,fill=black] (v12) ellipse (0.05 and 0.05);
\draw [black,fill=black] (v19) ellipse (0.025 and 0.025);
\coordinate (v20) at (-5,2.25) {} {} {};
\coordinate (v21) at (-5.75,2.25) {} {} {};
\coordinate (v22) at (-5.75,1.625) {} {};
\draw [-open triangle 45](v20) -- (v21) -- (v22);
\coordinate (v23) at (-5,0.25) {} {} {};
\coordinate (v24) at (-5.75,0.25) {} {} {};
\coordinate (v25) at (-5.75,0.875) {} {};
\draw [-open triangle 45] (v23) -- (v24) -- (v25);
\node[align=right,text width=0cm] (v26) at (-5.25,1.25) {\footnotesize\textit{~off}};
\draw [black,fill=black] (v26) ellipse (0.025 and 0.025);
\draw [dashed] (v19) -- (v26);
\coordinate (v27) at (-7,1.25) {};
\draw [-latex,thick](v27) -- (v18);
\draw [thick](v19) -- (v27);
\node at (-2.75,2.75) {\small{}Abstract controller for Phase 1};
\node at (-2.75,-0.25) {\small{}Abstract controller for Phase 2};
\node at (-5.25,2) {\small{}1};
\node at (-5.25,0.5) {\small{}2};
\end{tikzpicture}
\caption{\label{fig:controller}Approximately optimal controller for quantitative two-phase reach-avoid problems on sampled-data systems. The switch moves in the order ``1-2-\textit{off}" according \eqref{e:closedloopperformance:controller}, where ``\textit{off}" occurs when $\mu_2'$ signals stopping.}
\end{figure}
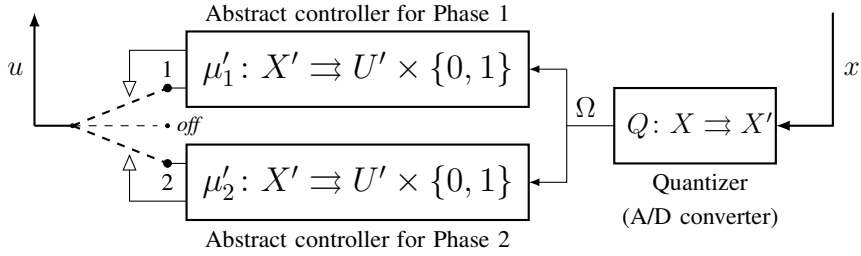
\fi
\begin{theorem}
\label{th:refinement}
Consider the algorithm in Fig.~\ref{alg:approximate} with input $\Pi = (X,U,F,G,g)$ and 
assume that it returns 
the controller $\mu_Q$. 
Then $\mu_Q$ 
approximately solves $\Pi$ in the sense that
\begin{equation}
\label{e:th:refinement}
L(p,\mu_Q) \leq \sup \{ V'(\Omega) \mid p \in \Omega \}
\end{equation} 
for all $p \in X$, 
where $L$ denotes the closed-loop performance of $\Pi$ and $V'$ stands for the value function 
of the abstract optimal control problem $\Pi'=(X',U',F',G',g')$ in line \ref{alg:approximate:apply}.
\end{theorem}
\begin{proof}
Let $\Pi_2$, $L_1$, $L_2$ be as in Proposition \ref{p:closedloopperformance}
when applied to $\Pi$ with $\mu_i := \mu_i' \circ Q$, $i \in \{1,2\}$.
Firstly, for $p \in X$ it follows that $L(p,\mu_Q) = L_1(p,\mu_1)$. 
Secondly, fix $\varepsilon>0$. There exists $(u,v,x) \in \mathcal{B}_p(\mu_1 \times S)$, where $S= (X,U,F)$, 
such that for $\tau = \min v^{-1}(1)$ the value $L_1(p,\mu_1)$ is less than\looseness=-1
\begin{equation}
\label{e:th:refinement:proof}
\sum_{t=0}^{\tau-1} g(x(t),x(t+1),u(t)) + L_2(x(\tau),\mu_2) + \varepsilon.
\end{equation}
Next, by \cite[Th.~V.4]{ReissigWeberRungger17} 
there exists $\Omega \in (X')^{\Z_+}$ satisfying a) 
$(u,v,\Omega) \in \mathcal{B}_{\Omega(0)}(\mu_1' \times S')$, 
b) $x(t) \in \Omega(t)$ for all $t \in \mathbb{Z}_+$ and
\ifx\arxivVersion\undefined
\linebreak{}
\fi
c) $L_2(x(\tau),\mu_2) \leq L_2'(\Omega(\tau),\mu_2')$, 
where $S'=(X',U',F')$ and $L_2'$ denotes 
the closed-loop performance of the optimal control problem $\Pi_2'$ defined 
in line \ref{alg:twophase:firstproblem} of Fig.~\ref{alg:twophase} when applied to $\Pi'$.  
Property c) follows from Theorem \ref{th:soc} applied to $\Pi_2$ and $\Pi_2'$.
Finally, apply Proposition \ref{p:closedloopperformance} to 
$\Pi'$, $\mu'_1$, $\mu_2'$ in place of 
$\Pi$, $\mu_1$, $\mu_2$ and denote by $L_1'$ and $L_2'$ the closed-loop performances appearing Prop.~\ref{p:closedloopperformance} for this case.
By Theorem \ref{th:soc} and properties b), c) above, \eqref{e:th:refinement:proof} 
is bounded by 
\ifx\arxivVersion\undefined
$\sum_{t=0}^{\tau-1} g'(\Omega(t),\Omega(t+1),u(t)) + L_2'(\Omega(\tau),\mu_2') + \varepsilon$. 
\else 
\begin{equation}
\sum_{t=0}^{\tau-1} g'(\Omega(t),\Omega(t+1),u(t)) + L_2'(\Omega(\tau),\mu_2') + \varepsilon.
\end{equation} 
\fi
The latter sum is less than $L_1'(\Omega(0),\mu_1') + \varepsilon$ by property a) above. Apply Theorem \ref{th:main} to $\Pi'$ and $\mu'_1$, $\mu'_2$ to see that the last sum equals $V'(\Omega(0)) +\varepsilon$. It follows $L_1(p,\mu_1) \leq V'(\Omega(0))$ by letting $\varepsilon \to 0$, which implies \eqref{e:th:refinement}.
\end{proof}
\section{Examples}
\label{s:examples}
The theoretical results are illustrated by experimental evaluations 
on two examples. 
In the first, the example already appeared in 
Section \ref{s:introduction} will be discussed in detail while in the second a multiple 
application of Fig.~\ref{alg:approximate} is considered. 
\subsection{Urban parcel delivery with vehicle model}
\label{ss:examples:vehicle}
\subsubsection{Control problem} 
A variant of the classical vehicle model \cite[Ch.~2.4]{AstromMurray08} is considered, 
whose dynamics is described by \eqref{e:system:cont} with 
$W=\{(0,0)\} \times \intcc{-\frac{1}{100},\frac{1}{100}} \times \intcc{-\frac{1}{10},\frac{1}{10}}$,
\begin{equation}
\label{e:examples:vehicle:dynamics}
f((x_1,\ldots,x_4),(u_1,u_2)) = \begin{pmatrix}
x_4 \cdot \cos(\alpha + x_3) \cdot \beta \\
x_4 \cdot \sin(\alpha + x_3) \cdot \beta \\
x_4 \cdot \tan(u_2) \\
u_1
\end{pmatrix},
\end{equation}
where $\alpha = \arctan(\tan(u_2)/2)$, 
$\beta =  \cos(\alpha)^{-1}$.
Here, $(x_1,x_2)$ describes the position, 
and $x_3$ and 
$x_4$ describe the orientation and 
the velocity of the vehicle, respectively. 
The control inputs $u_1$ and $u_2$
are the acceleration and the steering angle of the vehicle, 
respectively.
Formally, we consider the system $(\mathbb{R}^4,U,F)$
defined as the sampled system \cite[Def.~VIII.1]{ReissigWeberRungger17} 
associated with \eqref{e:system:cont} and sampling time $\tau = 0.1$, 
where $f$ is given by \eqref{e:examples:vehicle:dynamics}, 
$U = \intcc{-6,4} \times \intcc{-0.5,0.5}$.

The vehicle is located in the urban environment 
shown in Fig.~\ref{fig:intro}. 
We are going to solve 
the quantitative two-phase reach-avoid problem 
$\Pi$ of the form \eqref{e:ocp} 
associated with 
$A_1$, 
$A_2$ and 
$G_0$, 
where $G_0$ %
is the zero function and 
\begin{align*}
A_1 &= \intcc{33,41} \times \intcc{17,21} \times \intcc{-\pi,\pi} \times \intcc{0,7}, \quad \text{ (``Area 1")}\\
A_2 &= \intcc{51,55} \times \intcc{22,30} \times \intcc{-\pi,\pi} \times \intcc{0,7}. \quad \text{ (``Area 2")}
\end{align*}
The running cost $g$ satisfies $g(x,y,u) = \infty$ if $x$ violates the common right-hand traffic rules or is in the obstacle set 
$(\mathbb{R}^4 \setminus \bar X ) \cup O_1 \cup \ldots \cup O_4$ which restricts space and velocity. Here,
\begin{subequations}
\begin{align}
\bar X &=\intcc{0,64} \times \intcc{0,30} \times \intcc{-\pi,\pi} \times \intcc{0,18}, 
\label{e:examples:vehicle:X}
\\
O_1 &= \intcc{0,10} \times \intcc{0,11} \times \intcc{-\pi,\pi} \times \intcc{0,18},\\
O_3 &= \intcc{18,33} \times \intcc{8,22} \times \intcc{-\pi,\pi} \times \intcc{0,18}
\end{align}
\end{subequations}
and $O_2 = (0,19,0,0) + O_1$, $O_4 = (23,0,0,0) + O_3$.
The traffic rules are not violated if, e.g., $x$ is in the set
\begin{equation*}
\intcc{10,64} \times \intcc{0,4} \times \intcc{-3\pi/8, 3\pi/8} \times \intcc{0,18},
\end{equation*}
in which case $g$ compromises minimum time and proper driving style
by satisfying
\[
g(x,y,u) = \tau + u_2^2 + \min_{m \in M}\|(y_1,y_2) - m \|_2.
\]
Here, $\|\cdot \|_2$ is the Euclidean metric on $\R^2$ and 
$M \subseteq \mathbb{R}^2$ describes the axis of each roadway. 
E.g.
$
\intcc{12,62} \times \{2\} \subseteq M$.
\subsubsection{Approximate solution} 
To solve $\Pi$ approximately, 
we apply the procedure in Fig.~\ref{alg:approximate}. 
We use 
\cite[Sect.~VIII]{ReissigWeberRungger17} 
in combination 
with the technique in 
\cite[Sect.~IV]{WeberKreuzerKnoll20}
to compute a discrete abstraction $(X',U',F')$ 
for the sampled system, 
i.e. to perform 
line \ref{alg:approximate:abstraction} in 
Fig.~\ref{alg:approximate} and 
to solve the two abstract optimal control problems, 
i.e. lines \ref{alg:twophase:firstproblem} and \ref{alg:twophase:secondproblem} in Fig.~\ref{alg:twophase} invoked by 
line \ref{alg:approximate:apply} in Fig.~\ref{alg:approximate}.
The synthesis is successful. A closed-loop trajectory is shown in Fig.~\ref{fig:examples:vehicle}. 
\ifx\arxivVersion\undefined
\else
\begin{figure}
\begin{center}
\hspace*{1.3cm}
\input{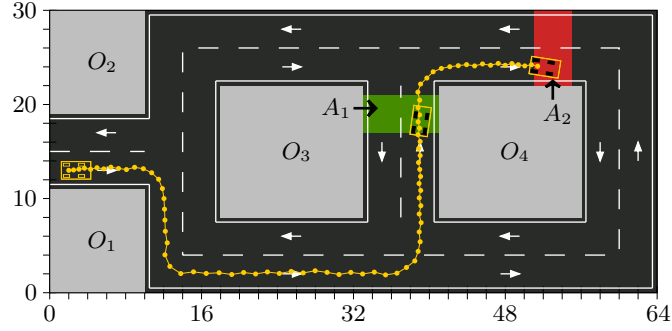}
\end{center}
\caption{\label{fig:examples:vehicle}Illustration of a closed-loop trajectory as obtained from interconnecting the plant in Section \ref{ss:examples:vehicle} with an approximately optimal controller for the given quantitative two-phase reach-avoid problem. The vehicle symbols indicate the initial state $(2,13,0,5)$ and the states where switching of controllers occurs (cf. Fig~\ref{fig:controller}), respectively.}
\end{figure}
\fi
The cost \eqref{e:costfunctional} for the shown trajectory is $31.78$ whereas the corresponding cost of the closed-loop trajectory shown in Fig.~\ref{fig:intro} is $24\%$ higher ($39.50$). 
Some more computational details are given in Tab.~\ref{tab:examples:vehicle}.
\begin{table}[h]
\centering
\begin{tabular}{|l|l|}
\hline
Quantity & Value (description) \\
\hline 
\hline
$|X'|$ & $\approx 21.2\cdot 10^6$ (subdivision of  \eqref{e:examples:vehicle:X} into\\
& $120\cdot 57\cdot 62 \cdot 50$ hyper-rectangles) \\
$|U'|$ & $99$ ($9\cdot 11$ values of $U$) \\
\hline 
Runtime line \ref{alg:twophase:firstproblem} in Fig.~\ref{alg:twophase} & $590$ sec. \\
Runtime line \ref{alg:twophase:secondproblem} in Fig.~\ref{alg:twophase} & $412$ sec. \\
Total runtime (Fig.~\ref{alg:approximate}) & $16.7$ min. \\
Total RAM usage & $9.9$ GB \\
\hline 
\end{tabular}
\caption{\label{tab:examples:vehicle}Computational details to Section \ref{ss:examples:vehicle}. 
Runtimes refer to a computation in parallel with 24 cores (Intel Xeon E5-2697 @ 2.6 GHz) on a Linux OS using an implementation of Fig.~\ref{alg:approximate} in C.}
\ifx\arxivVersion\undefined
\vspace{-3ex}
\else
\fi
\end{table}
\subsection{Routing mission with fixed-wing aircraft model}
\label{ss:examples:routing}
\ifx\arxivVersion\undefined
\else
\begin{figure}[b]
\centering
\includegraphics[scale=.92]{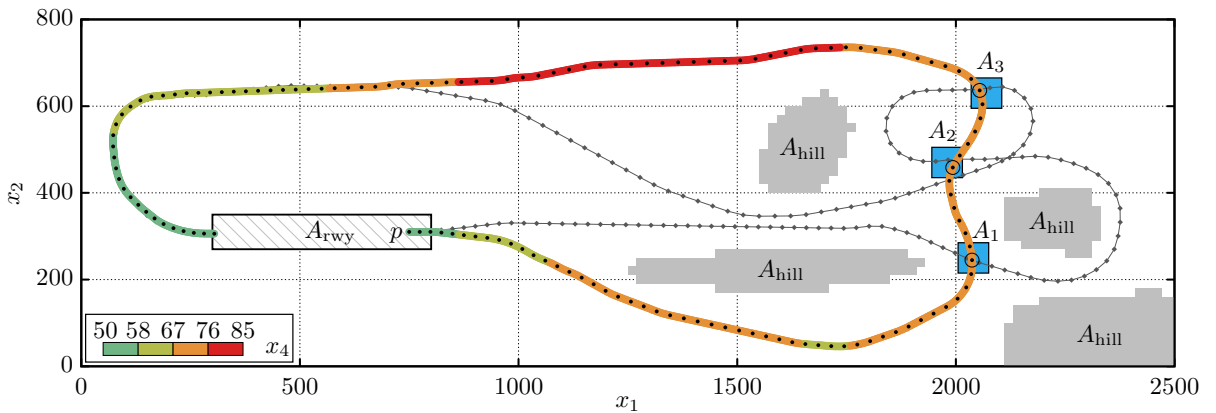}
\caption{\label{fig:waypointrouting}Routing mission from Section \ref{ss:examples:routing}. 
Closed-loop trajectories starting at $p = (750,310,0^\circ,52.5)$ are illustrated, where the multicoloured one is obtained from results of this work while the gray-coloured one is obtained from naively solving the given problem through four classical reach-avoid problems (Definition \ref{def:reachavoid}). At the circled points ($\odot$) of the multicoloured trajectory the switching to the next controller occurs (cf. Fig.~\ref{fig:controller}).}
\end{figure}
\fi
Next, the powerfulness of our results is demonstrated 
by considering a control problem for a fixed-wing aircraft, 
whose mission is to fly 
to three ``areas of interest" consecutively
after departure from the airfield 
and then back to it. A prescribed cost function is 
to be minimized.
Hence, 
a more general optimal control problem than in
Definition \ref{def:tworeachavoid} is considered and 
we will use the results of Section \ref{s:main}
formally as a heuristic only. Nevertheless, the result
is remarkable.

To anticipate the outcome straightaway, 
consider Fig.~\ref{fig:waypointrouting}:
The controller obtained by the results of this work 
requires for the mission $140$ time steps 
(the cost is $70.34$ as clarified later) 
leading to the multicoloured trajectory in Fig.~\ref{fig:waypointrouting}. 
In contrast, a sequence of ``locally" optimal controllers 
requires $184$ time steps (the cost is $36\%$ higher) 
leading to the entangled trajectory 
that is also indicated. 
The details to the control problem 
are given below.
\subsubsection{Control problem}
The plant is a four-dimensional fixed-wing aircraft model.
The states are the planar position $(x_1,x_2)$, 
heading and velocity ($x_3$ and $x_4$, respectively) of the aircraft. 
The inputs are the thrust and 
the bank angle ($u_1$ and $u_2$, respectively) of the aircraft.
The function $f$ in \eqref{e:system:cont} is given by $f_1$ in \cite[Eq.~14]{WeberKreuzerKnoll20}, see also \cite[Sect.~3]{GloverLygeros04}.
As in Section \ref{ss:examples:vehicle} 
we formally consider 
the sampled system $(\mathbb{R}^4,U,F)$ 
associated with \eqref{e:system:cont} and 
sampling time $\tau = 0.45$ 
with $W=\{0\}$, and 
$U = \intcc{0,18000} \times \intcc{-40^\circ,40^\circ}$. 

The sets defining the mission are given in Tab.~\ref{tab:routing}.
The trajectory cost \eqref{e:trajectorycost} is zero 
if the sets are visited in the order 
$A_1$, $A_2$, $A_3$, $A_\mathrm{rwy}$ while
the running cost \eqref{e:runningcost} is given by\looseness=-1
\begin{equation*}
g(x,y,u) = \begin{cases} 
\infty, & \text{if }x \in (\mathbb{R}^4 \setminus X_{\mathrm{mis}}) \cup O\\ \tau + u_2^2, & \text{otherwise ($u_2$ in radians)} 
\end{cases}.
\end{equation*}
Thus, the cost to minimize is a compromise 
between small bank angles and minimum time to accomplish mission.

To solve this mission, it is formalized by 
two quantitative two-phase reach-avoid problems 
$\Pi_1$ and $\Pi_2$ given by \eqref{e:ocp} with $G_i$ in place of $G$ defined as follows: 
\begin{itemize}[-]
\item $\Pi_2$ is associated with $A_3$, $A_\mathrm{rwy}$ and the zero function;
\item $\Pi_1$ is associated with $A_1$, $A_2$ and $V_2$, where $V_2$ is the value function of $\Pi_2$.
\end{itemize}
\begin{table}
\ifx\arxivVersion\undefined
\else
\centering
\begin{tabular}{|lll|}
\hline
Symbol & Value & Meaning \\
\hline 
\hline 
$X_{\mathrm{mis}}$ & $\intcc{0,2500} \times \intcc{0,800} \times \mathbb{R} \times \intcc{50,85}$ & Mission area and admissible heading and velocity \\
\hline
$A_1$ & $\intcc{2005,2075} \times \intcc{215,285} \times \times \mathbb{R} \times \intcc{50,75}$ & 1\textsuperscript{st} area of interest and admissible fly-over velocity\\
\hline
$A_2$ & $\intcc{1945,2015} \times \intcc{435,505} \times \mathbb{R} \times \intcc{50,75}$ & 2\textsuperscript{nd} area of interest and admissible fly-over velocity \\
 \hline
$A_3$ & $\intcc{2035,2105} \times \intcc{595,665} \times \mathbb{R} \times \intcc{50,75}$ & 3\textsuperscript{rd} area of interest and admissible fly-over velocity \\
 \hline
$A_\mathrm{rwy}$ & $\intcc{300,800} \times \intcc{270,350} \times \intcc{-10^\circ,10^\circ} \times \intcc{50,55}$ & Runway and admissible heading and velocity for landing\\
\hline 
$A_\mathrm{nofly}$ & $\intcc{320,780} \times \intcc{290,330} \times \intcc{12^\circ,348^\circ} \times \mathbb{R}$ & Illegal aircraft states over runway  \\
\hline
$A_\mathrm{hill}$ & $\subseteq \mathbb{R}^2$, see Fig.~\ref{fig:waypointrouting} & Spatial obstacle set \\
\hline 
$O$ & $(A_\mathrm{hill} \times \mathbb{R}^2) \cup A_\mathrm{nofly}$ & Overall obstacle set\\
\hline 
\end{tabular}
\fi
\caption{\label{tab:routing}Sets defining the optimal control problem in Section \ref{ss:examples:routing}.}
\ifx\arxivVersion\undefined
\vspace{-3ex}
\else
\fi
\end{table}
\subsubsection{Approximate solution}
For the approximate solution we use 
two abstract optimal control problems 
$(X',U',F',G'_i,g')$, $i \in \{1,2\}$, 
where the discrete abstraction 
(line \ref{alg:approximate:abstraction} 
in Fig.~\ref{alg:approximate}) is computed 
exactly as in \cite[Sect.~VI]{WeberKreuzerKnoll20}, e.g.
$|X'| \approx 141.7 \cdot 10^6$, $|U'| = 35$. 
The non-trivial cases in the definition of $G_1'$ and $G_2'$ equal the value function of $\Pi_2'$ and the zero function, respectively.
The total runtime and RAM usage for solving $\Pi_1'$ and $\Pi_2'$ is $3$ hours and $52$ GB using the technique in \cite{WeberKreuzerKnoll20} 
(hardware as in Section \ref{ss:examples:vehicle}).
\section{Conclusions and Outlook}
\label{s:conclusions}
From the theoretical point of view this work provides 
the structure of an optimal controller for solving 
two-phase reach-avoid problems on time-discrete nonlinear systems. 
In addition, a method to synthesize 
approximately optimal controllers 
on sampled-data systems has been deduced. 

Practicality of our results has been proved by simulations. 
In particular, our theoretical contributions 
proved to be useful also as 
a heuristic to solve approximately optimal control problems 
where several target sets must be visited.

Many open questions remain. 
Obviously, the structure of an optimal 
controller for several target sets is of interest 
as a generalization of Proposition 
\ref{p:closedloopperformance} and Theorem \ref{th:main}. 
Another open question is if a necessary condition for the optimality of such controllers can be formulated.
An even more general open problem is to generalize Theorem \ref{th:soc}
to optimal control problems with trajectory cost (rather than terminal cost). 
\bibliography{}
\end{document}